\documentclass[conference]{IEEEtran}
\IEEEoverridecommandlockouts
\usepackage{amsmath,amssymb,amsfonts,amsthm}
\usepackage{graphicx}
\usepackage{textcomp}
\usepackage{xcolor}
\usepackage{xspace}
\usepackage{booktabs}
\usepackage{algorithm}
\usepackage{algpseudocode}
\usepackage{xurl}
\usepackage{tabularx}
\allowdisplaybreaks[1]

\newtheorem{theorem}{Theorem}

\newtheorem{corollary}{Corollary}

\definecolor{lightgray}{gray}{0.9}
\definecolor{darkergray}{gray}{0.8}
\definecolor{forestgreen}{RGB}{34,139,34}

\usepackage{hyperref}
\hypersetup{hidelinks}
\usepackage[capitalize]{cleveref}

\DeclareMathOperator{\OPT}{OPT}

\begin{document}

\title{Fast and Accurate Hydro Unit Commitment}

\author{
\IEEEauthorblockN{Shaked Regev and Evan J. R. Brody}
\IEEEauthorblockA{
Oak Ridge National Laboratory\\
1 Bethel Valley Road\\
Oak Ridge, Tennessee, USA
}
\thanks{Notice: This manuscript has been authored by UT-Battelle, LLC, under contract DE-AC05-00OR22725 with the US Department of Energy (DOE). The US government retains and the publisher, by accepting the article for publication, acknowledges that the US government retains a nonexclusive, paid-up, irrevocable, worldwide license to publish or reproduce the published form of this manuscript, or allow others to do so, for US government purposes. DOE will provide public access to these results of federally sponsored research in accordance with the DOE Public Access Plan (\href{https://www.energy.gov/doe-public-access-plan}{https://www.energy.gov/doe-public-access-plan}). Research sponsored by the Laboratory Directed Research and Development Program of Oak Ridge National Laboratory, managed by UT-Battelle, LLC, for the US Department of Energy.}
}

\maketitle
 
\begin{abstract}
We introduce a fast novel algorithm to commit hydro-generators.  We show that our algorithm can commit 56 generators for 200 days 10,000 times faster than Hexaly, the leading commercial solver for scheduling problems, without sacrificing solution quality. Our method can commit 1792 generators for an entire year in approximately 20 seconds on commodity hardware. Hexaly is unable to solve problems of this size in reasonable time. We attain this runtime improvement by introducing a greedy-heuristic tailored for hydro units. We provide a direction for a possible proof  that the error our solver makes decreases with the problem size. This is supported by our experimental results.
We prove that our solver's runtime scales log-linearly in the number of generators and time periods. \end{abstract}
 

\section{Introduction}
\label{sec:introduction}
Hydropower is the main naturally replenishing, quickly dispatchable, and zero marginal cost generation source worldwide. Further, it can react quickly to respond to changing grid conditions and can start up with no supporting electricity, for example after a blackout~\cite{StanfordHydro}.  Worldwide it accounts for about 14\% of total electricity generation~\cite{lowcarbon}.


Co-optimizing hydro units with traditional, fuel-consuming units is essential to reduce generation costs while properly managing water levels. This is known as long term hydrothermal scheduling (\textbf{LTHS})~\cite{LTHS2014}. 
We focus on hydropower reservoirs, which are charged uncontrollably by rainfall and discharged controllably, subject to water balance constraints~\cite{Finardi2006835}.  We defer pumped hydropower, which functions as a battery with controllable charging and discharging~\cite{fenrg}, to future work. 

The main challenge in solving LTHS is that traditional units must be highly optimized over relatively short time scales to reduce costs. However, hydro units main cost is an opportunity cost, i.e. if they run now, they may not be able to run later. This means they must be optimized over long (typically yearly) time-horizons~\cite{LTHS2014,2017LTHS}. 

Existing work typically uses relatively slow LTHS solvers, which cannot run frequently. Therefore they require relatively accurate rainfall predictions or methods to deal with the inaccuracies of said predictions~\cite{CHAVES2025124730, WANG2026100701}. This introduces a tradeoff between optimizing for low generation costs and meeting contingencies in the mean and allowing for sufficient slack to account for inevitable large deviations from predictions\footnote{For example, in January-September 2026, rainfall was 8\% lower on average in the US compared to the same period in 2025. Some states received less than 75\% of 2026 rainfalls~\cite{NOAA2026}.}. 

Instead, knowing that forecasts will frequently fluctuate, we focus on developing a very fast solver for committing hydro units, which can be rerun every 5 minutes without contributing substantially to the overall time to solve LTHS. Our solver can consider long term predictions, but is agile enough to adapt as they change. It is designed primarily for grids where hydropower accounts for a relatively small percentage of total electricity generation, such as the US, where it is about 6\%~\cite{EIAmakeup}. The main contributions of this paper are:
\begin{itemize}
    \item Introducing a fast, scalable log-linear algorithm for scheduling hydro units as an alternative to existing exponential ones. 
    \item Showing that in all of our tests, our solvers produce solutions with identical objective values to a commercial solver in much less time. Our solutions are always feasible, with no constraint violations, and within tolerance of the lower bound on the optimum.
    \item Providing some analytical guarantees on the accuracy and runtime of our  algorithms.
\end{itemize}
Some limitations on the scope of this study include:
\begin{itemize}
\item Our grid model does not include transmission constraints. One can run our algorithm on a zonal level to set target production periods for the hydro units. Then one can add these targets to the objective function and solve a full alternating current optimal power flow (\textbf{ACOPF}) problem~\cite{brody2026fastrelaxandroundunitcommitment} to guarantee feasibility.  
    \item Our solver is not a general solution for all integer programs or scheduling problems.

\end{itemize}

The paper is organized as follows: \cref{sec:mathhydro} introduces our long-horizon planning hydro unit commitment (\textbf{UC}) problem formulation and algorithm.
\cref{sec:details} explains the details of the data acquisition and simulation setup. \cref{sec:results} compares our solver with the commercial solver Hexaly~\cite{hexaly_optimizer_2026}.
\cref{sec:summary} summarizes our results and future research directions. 


\section{Mathematical Formulation}
\label{sec:mathhydro}
We now introduce a mathematical formulation for the problem of committing  hydro units. This approach is intended to be used to commit hydro units before committing traditional units. This decoupling is justified because hydro units recharge by natural water flow and cannot indefinitely hoard water due to dam capacity. So as long units maintain high water levels long term, there is little opportunity cost. 

Hydro units are fast ramping \cite{Kincic} and typically have a narrow efficient operation regime \cite{KongIlievSkjelbred2024LifetimeHydraulicTurbineCost}. Here, we make the simplifying assumptions that hydro units can only be on at max capacity or off, the starting water levels are far away from $0$ and the maximum dam capacity, and water evaporation rates are independent of water levels. 

We denote by $W_i([T])$ the amount of water projected to flow into reservoir $i$ over a certain time window $[T]\doteq\{1,\ldots,T\}$. We denote by $\tilde{W}_i$ the water required for hydro unit $i$ to produce electricity in one time period. The number of time periods it can operate within window $[T]$ is $\pi_i \doteq \min(W_i([T])/\tilde{W}_i, T)$. This reduces our problem to finding in which periods $t\in [T]$ to operate each unit.

Hydro generators not being able to run all of the time (or even most of it) means they require optimizing over long-term, for example yearly, cycles. We therefore design a method to commit hydro generators based on shifting predicted demands.  For this purpose, we introduce the notion of the marginal cost of the system.

The marginal cost of a UC of traditional generators is the cost difference to produce slightly more or less power. Consider a cost model unit $i$'s production
\begin{equation}
\label{eq:traditional_gen}
u_i(a_iP_i^2+b_iP_i+c_i),
\end{equation}
where $u_i=0$ if the unit is off and $1$ if it's on, $P_i$ is the power produced by the unit, and $a_i, b_i, c_i$ are cost coefficients. If $u_i=0$, unit $i$'s marginal cost is undefined (it cannot produce slightly more or slightly less power). If $u_i=1$, the marginal cost of a unit with the production profile of \cref{eq:traditional_gen} is its derivative with respect to $P_i$, $2a_iP_i+b_i$.

Each $P_i$ is limited between some maximum and minimum due to production and ramping capabilities. All on units not at the feasibility boundary must have the same marginal cost $(2a_iP_i+b_i)$ at the optimal solution. Otherwise, shifting one unit's production up and one down would reduce cost. We define the marginal cost of the system as $(2a_iP_i+b_i)$ for one such of these units $i$.

We derive marginal costs from the solutions to a UC problem with traditional generators~\cite{regev2026economicUC}.
\cref{fig:marginal} shows how commonly certain demands coincide with certain marginal costs. A high demand always leads to high marginal costs and low demand always leads to low costs. There are small deviations in this trend due to periods that are less constrained by future generation requirements and ``free'' power from ramping generators. 

\begin{figure}[htbp]
\centering
\includegraphics[width=1.0\linewidth]{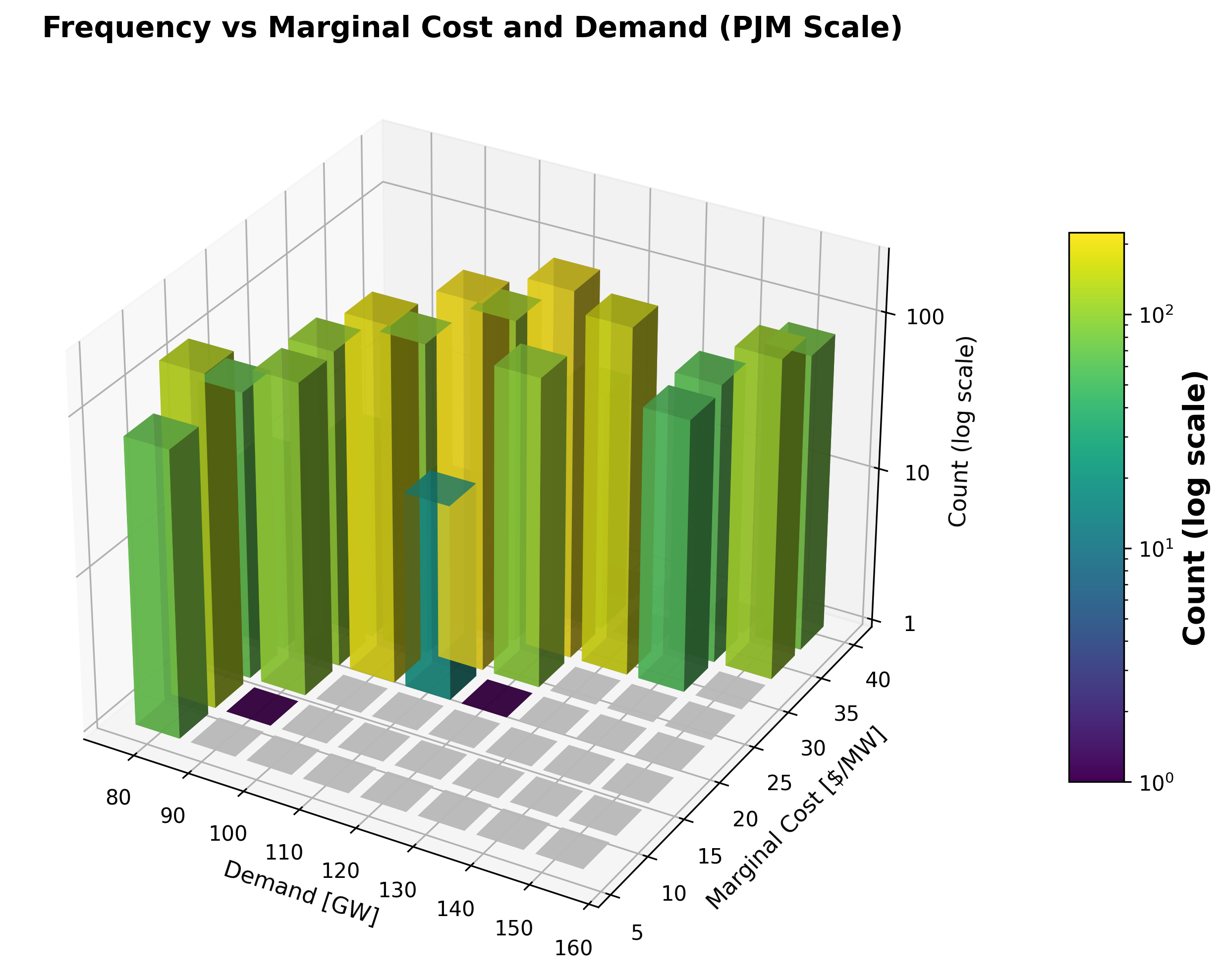}
\caption{\label{fig:marginal}Histogram of the frequency of demand-marginal cost pairs (counts) at the PJM scale. Marginal cost and demand are highly positively correlated. All rows and columns have at most 3 adjacent nonzero boxes, most have 2. This shows that outliers do not deviate much.}
\end{figure}

\cref{fig:marginal_fits} shows 1, 2, and 3 degree polynomial fits of the marginal cost to the demand, with their adjusted R-squared $\Bar{R}^2$. The $\approx 2\times$ difference between minimum and maximum demand to $\approx 8 \times$ difference between minimum and maximum marginal cost further motivates a cubic fit. A 4th degree increased $\Bar{R}^2$ less than $0.1\%$ and made no visual difference. 

\begin{figure}[htbp]
\centering
\includegraphics[width=1.0\linewidth]{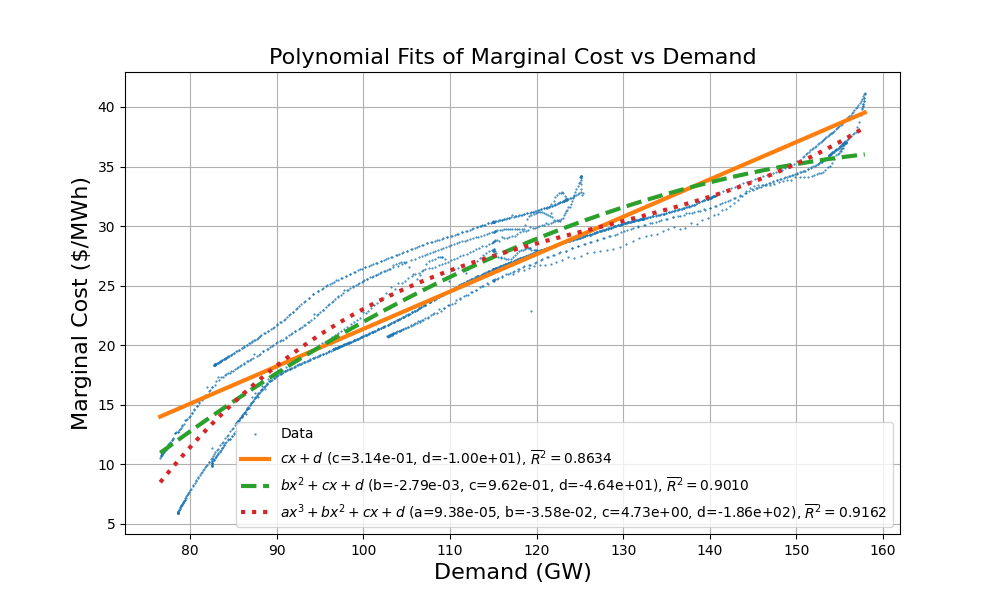}
\caption{\label{fig:marginal_fits} Polynomial fits of marginal cost as a function of demand. Marginal cost and demand are highly positively correlated, with a few outliers. The degree 3 polynomial has the largest adjusted R-squared ($\Bar{R}^2$) and is visually most similar to the data. The relative scale of the axes also suggests a cubic fit. The linear term is always dominant and positive. This suggests it is almost always better to shift net demand from high demand periods to low demand periods with hydropower or storage.}
\end{figure}

The linear term is dominant in all the fits shown here. The linear term of the 4 degree fit has a non-dominant negative coefficient, a clear sign of over-fitting, based on visual inspection of the data. 

This fit motivates the heuristic of using hydropower to prefer reducing demand from high demand periods over low demand periods. This is also known as load flattening or peak shaving. It also shows that even if pumped storage (or any other battery) does not retain energy perfectly, i.e. it can discharge a smaller amount than it charges, it can still be well worth-while to shift demand from low to high demand periods.

We define the demand remaining after hydro-generator commitment at time $t$ optimized over $T$ time periods as $\Bar{D}_{t}=D_t-\sum_{i=1}^nu_{i,t}\kappa_i$, where $\kappa_i$ is the power hydro-generator $i$ can produce. Smoothing out $\Bar{D}_{t}$ will lead to lower overall costs, because the marginal cost is highest when demand is highest. This leads to the following optimization problem: 

\begin{align}
\min_{u_{i,t}} \quad &
\text{Var}\left(D_t-\sum_{i=1}^nu_{i,t}\kappa_i\right)
\label{eq:HydroCommit}
\\
\text{s.t.} \quad &\forall 1\leq i\leq n: \sum_{t=1}^T u_{i,t} = \pi_i, u_{i,t}\in \{0,1\}
\nonumber
\end{align}
In a production setting, one would supplement this solve with a fast full ACOPF solve~\cite{brody2026fastrelaxandroundunitcommitment}, which could potentially change the commitments of the hydro units. This could be combined with new rainfall forecasts.
\begin{algorithm}[htbp]
\caption{Demand Balancing}
\label{alg:hydrocommit}
\begin{algorithmic}[1]
\Require Capacities $\kappa_i$, maximum operating periods $\pi_i$ for $i=1,\dots,n$; demands $d_t$ for $t=1,\dots,T$
\State $G_t \gets 0 \quad \forall t$ \Comment{Generated hydro-power at period \(t\)}
\State Initialize indexed max-heap $H$ with keys $k_t = D_t$
\State $\text{used}_t \gets 0 \; \forall t$ 
\Comment{Tracks assignments}

\State Re-index such that $\kappa_1 \sqrt{\pi_1}\geq\ldots \kappa_n\sqrt{\pi_n}$ \Comment{Improves numerical stability}
\For{\(j\gets 1\)\textbf{\ to\ }\(n\)}
    \State $\text{placed} \gets 0$
    \While{$\text{placed} < \pi_i$}
        \State $\tau \gets \mathrm{argmax}_{t} \; k_t$
        \If{$\text{used}_{\tau} = j$}
            \State Set $k_\tau \gets \infty$ and update $H$
        \Else
            \State $\text{used}_\tau \gets j$
            \State $\text{placed}\gets\text{placed}+1$,
            \State \(G_\tau\gets G_\tau + \kappa_i\)
            \State Update $k_\tau \gets D_\tau - G_\tau$ in $H$
        \EndIf
    \EndWhile
    \For{\(t\gets1\)\textbf{\ to\ }\(T\)}
        \If{$k_t = \infty$}
            \State Update $k_t \gets D_t - G_t$ in \(H\)
        \EndIf
    \EndFor
\EndFor
\State \Return $(L_t), (D_t - G_t), \mathrm{Var}(D_t - G_t)$
\end{algorithmic}
\end{algorithm}
Solving \cref{eq:HydroCommit} exactly is NP-hard (see \cref{thm:hydro np-hard} in Appendix \ref{sec:AppendixB}). \cref{alg:hydrocommit} approximates a solution by ``greedily'' assigning a generator which still has remaining periods to the period with the highest remaining demand. This is equivalent to locally optimizing \cref{eq:HydroCommit}. \cref{alg:hydrocommit} runs in log-linear time (see \cref{thm:runtime_hydro} in Appendix \ref{sec:AppendixB}).

\section{Simulation Details}
\label{sec:details}
We use hydro unit data from the US Energy Information Administration (\textbf{EIA})~\cite{eia8602024}. From this set, we take all hydro units in PJM. In this work we consider all hydro-generators as reservoirs, as opposed to pumped storage, which we save for future work. We aggregate the nameplate capacities for each plant from~\cite{eia8602024}. The capacity factor is the percentage of time a hydro-generator can operate. 

In the contiguous US, the median capacity factor is $36-39\%$, with large variance~\cite{Uria-Martinez2021USHydropowerMarket}. In the absence of more granular data, we assign different capacity factors between $28$ and $47\%$ to the generators to test \cref{alg:hydrocommit}. This gives us a base set of $56$ hydro-generators for PJM's demand. For our scaling experiments, we multiply the total demand by a factor $N$, and create $N-1$ noisy copies of the generators to commit. The copies have a capacity that shifted up to $\pm 10\%$ from the original generator, and can operate up to $\pm 10$ periods within a $24$ hour time window compared to the original generator. 

 Due to the long run times required for comparison over  thousands of time periods at many different scales, we run \cref{alg:hydrocommit} and our comparisons on a particular dataset. We scale this dataset and perturb it randomly to ensure our method is robust. This is the limit of our scope.

\section{Results}
\label{sec:results}
We compare our specialized solver for \cref{eq:HydroCommit} to the commercial integer program solver Hexaly~\cite{hexaly_optimizer_2026}. Hexaly specializes in general integer optimization scheduling problems with many constraints, such as \cref{eq:HydroCommit}.

\cref{fig:hydro_relative} compares the relative differences between the solutions produced by
\cref{alg:hydrocommit} and a full integer program solve with Hexaly with the lower bound on the optimum. \cref{alg:hydrocommit} consistently has a relative error below $10^{-7}$. Hexaly always produces solutions with a lower error than the $10^{-4}$ tolerance. This was true even when we tested it with stricter tolerances.  Hexaly reported timeout errors (even before it met the time limit) for problems larger than in \cref{fig:hydro_relative}.

\begin{figure}[htbp]
\centering
\includegraphics[width=1.0\linewidth]{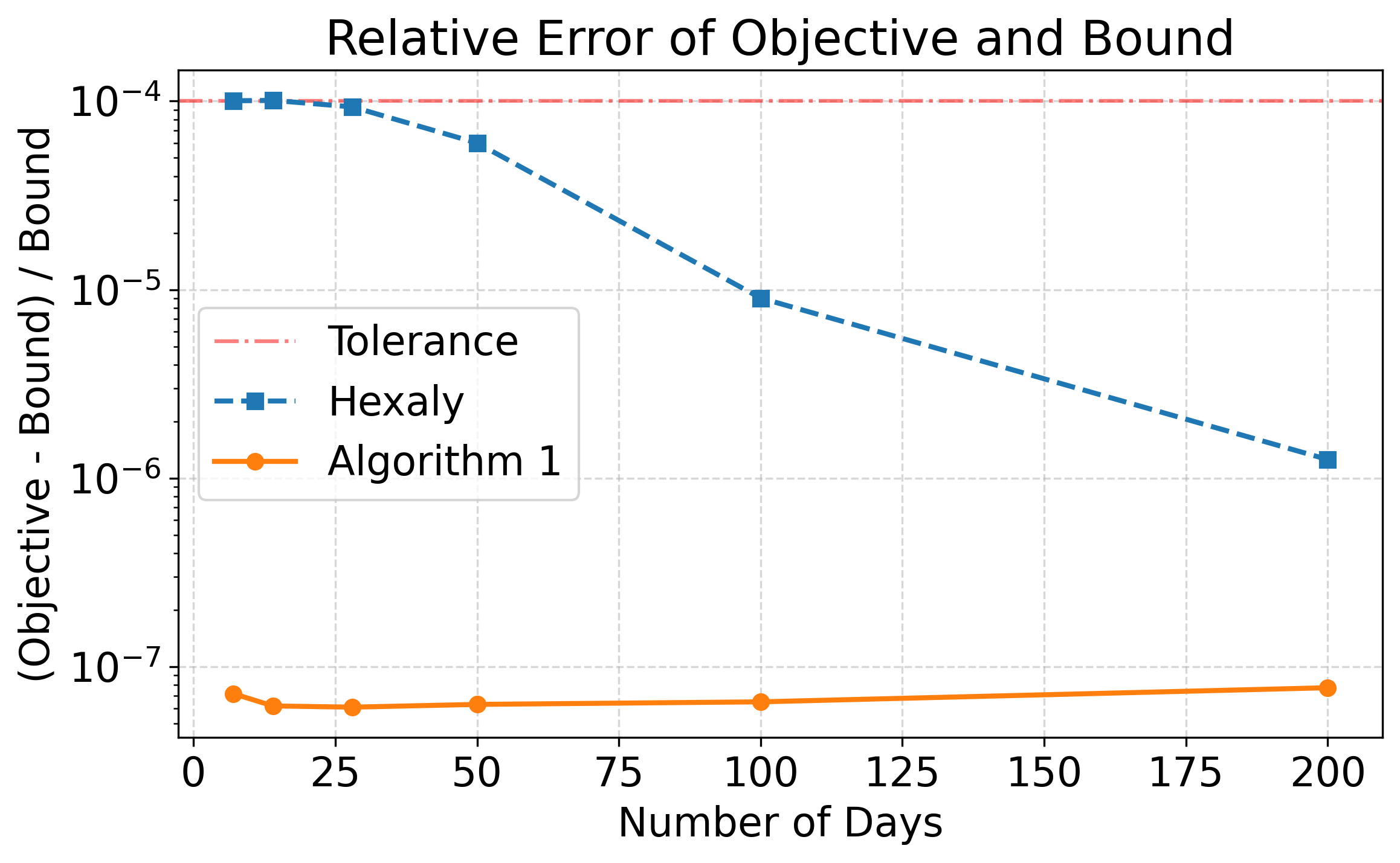}
\caption{\label{fig:hydro_relative} \cref{alg:hydrocommit} difference from the lower bound on the optimum compared to Hexaly integer program solver for committing 56 generators for different time horizons. Both solvers consistently produce solutions with objective deviations within tolerance of the lower bound.} 
\end{figure}

\cref{fig:hydro_compare} compares run times of \cref{alg:hydrocommit} and Hexaly. \cref{alg:hydrocommit} is $1000$ times faster on the smallest problem we tested
and this gap grew with the problem size.

\begin{figure}[htbp]
\centering
\includegraphics[width=1.0\linewidth]{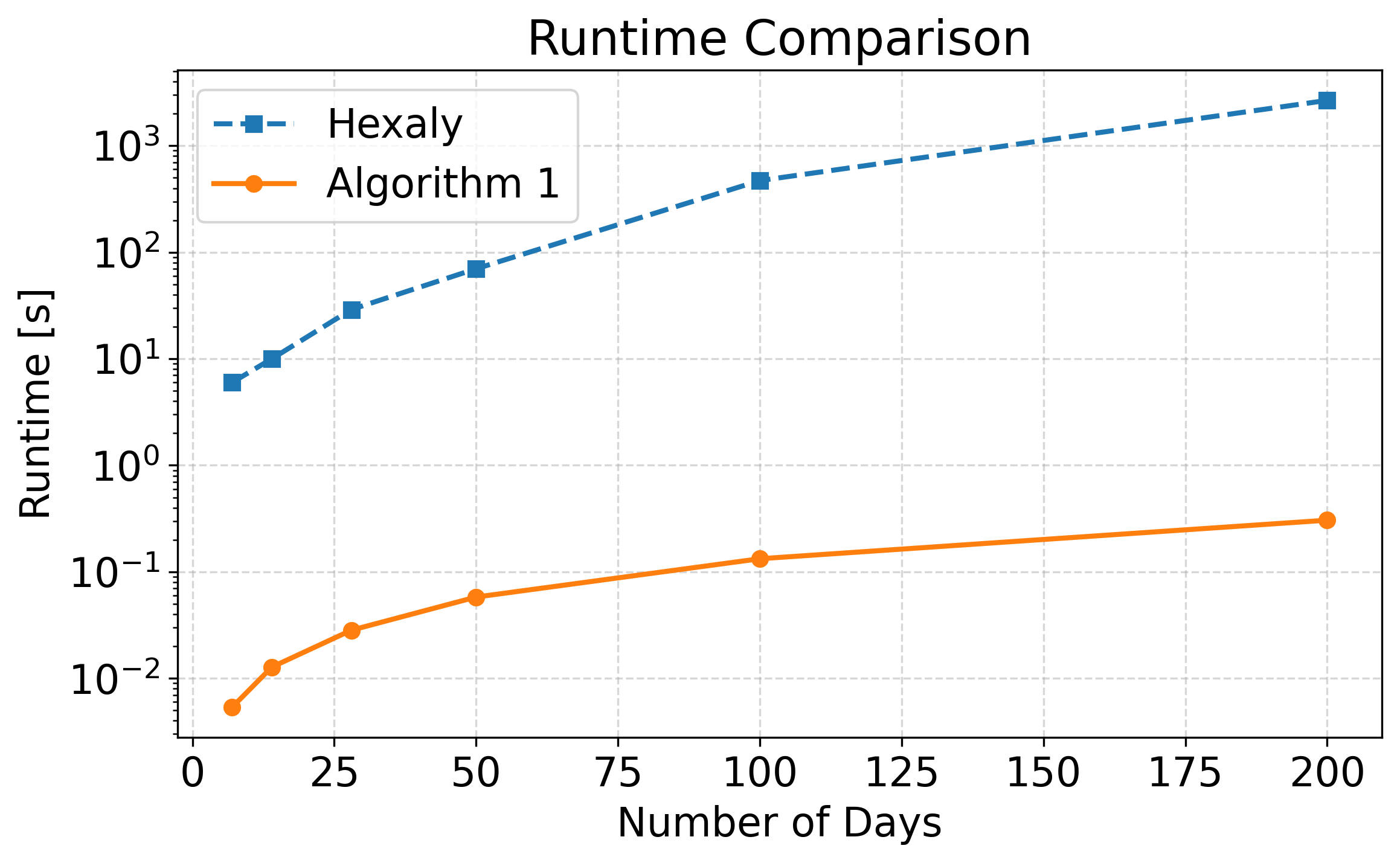}
\caption{\label{fig:hydro_compare} \cref{alg:hydrocommit} runtime compared to Hexaly integer program solver for committing 56 generators for different time horizons.  \cref{alg:hydrocommit}  gets much faster solutions, a gap that increases roughly linearly in the problem size. It was almost $10,000$ times faster on the largest problem shown here.} 
\end{figure}

\cref{fig:hydro_commit} shows normalized $\sigma_N\propto \sqrt{\text{Var}(\Bar{D}_\mathcal{T})}/m$, where $m$ is the number of generators. $\sigma_N$ is normalized so that with $56$ generators allocated over 7 days, $\sigma_N=1$.  
\begin{figure}[htbp]
\centering
\includegraphics[width=1.0\linewidth]{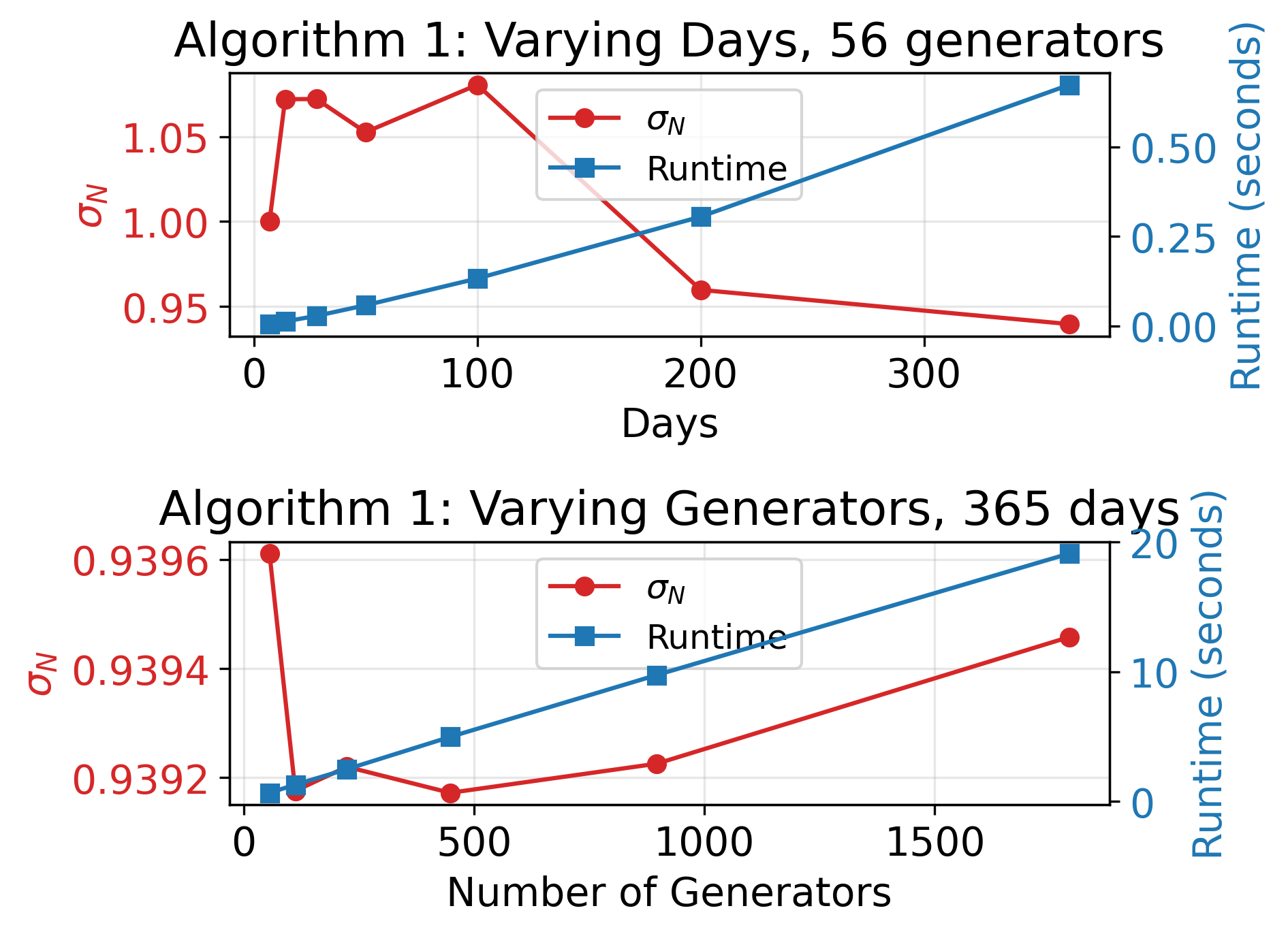}
\caption{\label{fig:hydro_commit} \cref{alg:hydrocommit}'s normalized objective and runtime varying the number of days and generators. Its performance is stable as the demand (and number of generators to supply it) increase and in optimizing over longer stretches. Its runtime increases linearly in the number of generators, and slightly faster than linearly in the number of days (or periods) as \cref{thm:runtime_hydro} guarantees.}
\end{figure}

\cref{alg:hydrocommit} retains accuracy in its objective and its run time scales almost linearly. It can commit 1792 generators for an entire year in 20 seconds. This problem is 60 times larger than the largest problem Hexaly could run. This is useful because rainfall follows yearly cycles. So it can easily run at 5-minute intervals to keep up with evolving demand prediction and dam water levels (which impact how many periods within the year they can operate). The run time is short enough such that an operator can rerun \cref{alg:hydrocommit} whenever demand or rainfall forecasts change.

\section{Conclusions and Further Work}
\label{sec:summary}
We introduce an algorithm to commit many hydro-generators over many periods. We analytically prove that it scales log-linearly in the number of generators and periods. We argue that a bound on the solution's distance from optimality must become tighter as the problem grows. We verify our algorithm by checking its  objective and its scaling in practice. 

Our algorithm can commit 1792 generators in 5 minute periods for a year in less than 10 seconds. Its experimental scaling matches the analytical runtime analysis. This means that it is fast enough to rerun at 5 minute intervals when new rainfall predictions or new constraints come in on a American interconnect scale. It is almost 10,000 times faster than Hexaly on the largest problem Hexaly could solve, committing 56 generators over 200 days. 

All solutions satisfy all constraints and are within $10^{-7}$ relative tolerance of lower bounds on the optimum and within $10^{-4}$ relative tolerance of solutions produced by Hexaly. This accuracy is substantially higher than our predicted analytical bound. 
Our algorithm retains accuracy at scales we tested.

In future research, we will incorporate transmission constraints in a subsequent solver step that enforces them strictly. The first step will incorporate the time constraints implicitly by adding a penalty to the objective function in the full ACOPF solve (the second step). Using the insights of \cref{fig:marginal,fig:marginal_fits}, we will incorporate batteries, including hydro-storage.

\section*{Acknowledgements}
 We thank Slaven Pele\v{s} for his technical review. We thank Eve Tsybina and James Nutaro for helping write the proposal that funded this work.
\appendices
\section{Asymptotic optimality}
\label{sec:AppendixA}

To bound \cref{alg:hydrocommit}'s distance from optimality, we define $\Bar{D}^*_\mathcal{T}$, the optimal solution to \cref{eq:HydroCommit}, and $K\doteq\max_i(\kappa_i)$. For any time period $t$, we define the error  \cref{alg:hydrocommit} makes as $\Delta_t$. So
$\Bar{D}_{\mathcal{T},t}+\Delta_t\doteq \Bar{D}^*_{\mathcal{T},t}$, with 
$\sum_t\Delta_t=0$. 
\begin{theorem}
    \label{thm:optimality} If $\forall t, \;|\Delta_t|\leq K$\footnote{It's intuitive that any error 
cannot exceed the largest generator capacity. Otherwise \cref{alg:hydrocommit} could simply place that generator differently at the corresponding time step. We were unable to show $\forall t, \;|\Delta_t|\leq K$ rigorously. We found solutions that meet this bound, but not ones that exceed it. }, then
   $\text{Var}(\Bar{D}_\mathcal{T})-\text{Var}(\Bar{D}^*_\mathcal{T})\leq K^2+2K\sqrt{\text{Var}(\Bar{D}^*_\mathcal{T})}$. 
\end{theorem}
\begin{proof}
Let the mean remaining demand 
be $\mu\doteq \sum_t \Bar{D}_\mathcal{T}/T$, where $T$ is the number of periods. $\mu$ is identical for $\Bar{D}^*_\mathcal{T}$ and solutions produced by 
\cref{alg:hydrocommit} (or by any algorithm that assigns all units to their maximum number of periods).

\begin{align}
\label{eq:vardif}
&\text{Var}(\Bar{D}_\mathcal{T})-\text{Var}(\Bar{D}^*_\mathcal{T})
=\frac{1}{T}\sum_{t=1}^T [(\Bar{D}^*_{\mathcal{T},t}+\Delta_t-\mu)^2
 \nonumber \\
&- (\Bar{D}^*_{\mathcal{T},t}-\mu)^2] =\frac{1}{T}\sum_{t=1}^T \left[2\Delta_t(\Bar{D}^*_{\mathcal{T},t}-\mu)+\Delta_t^2\right].
\end{align}

Using the Cauchy–Schwarz inequality,
\begin{align}
&\frac{1}{T}\sum_{t=1}^T \Delta_t(\Bar{D}^*_{\mathcal{T},t}-\mu)
\leq \sqrt{\frac{1}{T}\sum_{t=1}^T \Delta_t^2}
\sqrt{\frac{1}{T}\sum_{t=1}^T (\Bar{D}^*_{\mathcal{T},t}-\mu)^2},\nonumber \\ \label{eq:vardif_final}
&\text{Var}(\Bar{D}_\mathcal{T})-\text{Var}(\Bar{D}^*_\mathcal{T})
\leq K^2+2K\sqrt{\text{Var}(\Bar{D}^*_\mathcal{T})}.
\end{align}

The last line uses $\frac{1}{T}\sum_{t=1}^T \Delta_t^2\leq K^2$ in 
\cref{eq:vardif}.
\end{proof}

We can rewrite the result of \cref{thm:optimality} as
\begin{align}
\label{eq:varbound}
&\text{Var}(\Bar{D}^*_\mathcal{T})
\leq \text{Var}(\Bar{D}_\mathcal{T})
\leq \text{Var}(\Bar{D}^*_\mathcal{T})
+ K^2 + 2K\sqrt{\text{Var}(\Bar{D}^*_\mathcal{T})} \nonumber \\
&\Rightarrow
1\leq
\frac{\text{Var}(\Bar{D}_\mathcal{T})}{\text{Var}(\Bar{D}^*_\mathcal{T})}
\leq
1 + \frac{K^2}{\text{Var}(\Bar{D}^*_\mathcal{T})}
+ \frac{2K}{\sqrt{\text{Var}(\Bar{D}^*_\mathcal{T})}}.
\end{align}

\begin{corollary}
\label{cor:optimal}
Under our assumptions and for our purposes, where load variation is much larger than the available hydropower 
($\text{Var}(\Bar{D}^*_\mathcal{T})\gg K^2$), \cref{alg:hydrocommit} produces 
nearly optimal solutions.
\end{corollary}
Allowing hydro units to produce at partial capacity is equivalent to increasing $m$ and having fractional capacities $\kappa_i$.
\begin{corollary}
\cref{thm:optimality} and \cref{cor:optimal} hold even if our units can produce at partial capacity, because our proof makes no use of $m$ and $K$ does not increase. If the largest unit's production is split among periods (decreasing $K$), we get a constant factor improvement, with similar asymptotics.
\end{corollary}

\section{Asymptotic run time}
\label{sec:AppendixB}

\begin{theorem}
    Solving \cref{eq:HydroCommit} is NP-hard.
    \label{thm:hydro np-hard}
\end{theorem}
\begin{proof}
    We reduce from the partition problem, which is NP-complete \cite{Karp1972}. Given a vector \((c_1,\ldots,c_n)\in\mathbb{Z}^n\), the partition problem asks whether there exists \(I\subseteq\{1,\ldots,n\}\) satisfying
   \begin{equation}
           \sum_{i\in I}c_i = \sum_{i\notin I}c_i
        \label{eq:partition}   
   \end{equation}
    Without loss of generality, take \(c_i\geq 0\) for all \(i\).\par
    Given an instance of the partition problem \cref{eq:partition}, consider an instance of \cref{eq:HydroCommit} where \(T=2\), there are \(n\) generators, and \(\kappa_i=c_i, \pi_i = 1\) for all \(i\). Set \(D_1 = D_2=\sum_i c_i\). For this instance of \cref{eq:HydroCommit}, \(\OPT = 0\) if and only if \(D_1 - \sum_i\kappa_iu_{i,1} = D_2 - \sum_i\kappa_iu_{i,2}\) if and only if \(\sum_i\kappa_iu_{i,1} = \sum_{i}\kappa_iu_{i,2}\). This is possible if and only if there exists a set \(I\) satisfying the condition of the partition problem (i.e., \(I=\{1\leq i \leq n : u_{i,1} = 1\}\)).
\end{proof}
\begin{theorem}
\label{thm:runtime_hydro}
\cref{alg:hydrocommit} takes $O(n(\log n +T\log T))$ time, where $n$ is the number of hydro units and $T$ is the number of periods over which to optimize.
\end{theorem}
\begin{proof}
Lines $1-4, 24$ run in $O(n\log n+T)$ time.  The outer for loop (line \(5\)) runs $n$ times. The inner while loop on line \(7\) runs $O(T)$ times, since it considers a different period on each iteration. The inner for loop (line 18) runs $T$ times. All lines inside loops run in $O(1)$ time, except lines $9$ and $16$, which run in $O\left(\log T\right)$~\cite{CLRS2009}. Overall, \cref{alg:hydrocommit} runs in
$
    O\left(\underbrace{n\log n + T}_{\text{lines \(1-4,24\)}} + \underbrace{n(T\log T + T)}_{\text{lines \(5-23\)}}\right) = O(n(\log n +T\log T))
$
time.
\end{proof}

\bibliographystyle{IEEEtran}
\bibliography{bibfile}
\end{document}